\documentclass[reqno,12pt]{amsart}
\usepackage{amssymb}
\usepackage{mathrsfs}
\usepackage{txfonts}
\usepackage{amsfonts}
\usepackage{amstext}
\usepackage{amssymb}
\usepackage{amsmath}
\usepackage{graphicx}
\usepackage{hyperref}
\usepackage{url}
\usepackage{amssymb}
\usepackage{subfig}

\hypersetup{
        colorlinks   = true,
        citecolor    = blue,
        linkcolor    = blue,
        urlcolor     = blue
}
\allowdisplaybreaks
\newtheorem{theorem}{Theorem}[section]
\newtheorem{proposition}[theorem]{Proposition}
\newtheorem{lemma}[theorem]{Lemma}

\renewcommand{\theequation}{\thesection.\arabic{equation}}

\numberwithin{equation}{section}
\newcounter{counterConstant}

\input{tcilatex}

\def\Xint#1{\mathchoice
{\XXint\displaystyle\textstyle{#1}}%
{\XXint\textstyle\scriptstyle{#1}}%
{\XXint\scriptstyle\scriptscriptstyle{#1}}%
{\XXint\scriptscriptstyle\scriptscriptstyle{#1}}%
\!\int}
\def\XXint#1#2#3{{\setbox0=\hbox{$#1{#2#3}{\int}$ }
\vcenter{\hbox{$#2#3$ }}\kern-.6\wd0}}

\def\oint{\Xint-}

\def\FRAME#1#2#3#4#5#6#7#8
{
 \begin{figure}[ptbh]
 \begin{center}
 \includegraphics[height=#3]{#7}
 \caption{#5}
 \label{#6}
 \end{center}
 \end{figure}
}

\def\enddoc{\end{document}}

\begin{document}
\title[]{On the minimal generating weighted IFS of self-similar measure}
\author[Zhang]{Junda Zhang}
\address{School of Mathematics, South China University of Technology,
Guangzhou 510641, China.}
\email{summerfish@scut.edu.cn}

\begin{abstract}
We concern the structrue of generating weighted IFSs of a self-similar measure on the real line.
We provide various sufficient conditions for the existence of a minimal generating weighted IFS of a self-similar measure on the real line.
Under the homogeneity, we show that `most' self-similar measures on the real line have a minimal generating weighted IFS, without separation conditions.
The ingredients of our proofs are based on results of exponential polynomials (factorization theory and the distribution of zeros),
logarithmic commensurability (with a dynamical system argument), and results on the structrue of generating IFSs of a self-similar sets.
\end{abstract}

\date{\today}
\maketitle
\tableofcontents

\section{Introduction}

    Self-similar measures form a fundamental class of fractal measures that arise naturally in the study of dynamical systems, harmonic analysis, and fractal geometry. They were introduced in a systematic way by Hutchinson~\cite{Hutchinson81} and have since been a central object of investigation. In this paper, we focus on standard self-similar measures on the real line (see recent progress in \cite{varju2023selfsimilar}).

    Let $N \geq 2$, and consider a standard \emph{iterated function system} (IFS) consisting of distinct contractive similitudes in $\mathbb{R}$:
\begin{equation}\label{eq:IFS}
    S_j(x) = r_j x + b_j, \qquad j = 1,\dots,N,
\end{equation}
where $0 < |r_j| < 1$ and $b_j \in \mathbb{R}$. Assign to each map $S_j$ a probability weight $p_j > 0$ such that $\sum_{j=1}^{N} p_j = 1$. We call the IFS $\{S_j\}_{j=1}^{N}$ associated with the probability vector $(p_1,\dots,p_N)$ a \emph{weighted IFS}.
A \emph{self-similar measure} associated with this weighted IFS is a Borel probability measure $\mu$ on $\mathbb{R}$ satisfying the invariance equation
\begin{equation}\label{eq:inv}
    \mu = \sum_{j=1}^{N} p_j \, \mu \circ S_j^{-1}.
\end{equation}
Hutchinson~\cite{Hutchinson81} proved that there exists a unique such measure, and its support is the \emph{attractor} of the IFS, namely, the unique non-empty compact set $K$ with \begin{equation}\label{at}K = \bigcup_{j=1}^{N} S_j(K).\end{equation} We call the IFS $\{S_j\}_{j=1}^{N}$ a \emph{generating IFS} of $K$, and the pair $(\{S_j\}_{j=1}^{N},(p_1,\dots,p_N))$ a \emph{generating weighted IFS} of $\mu$.
Throughout this paper, whenever we mention an IFS, we assume that it is standard (consisting of contractive similitudes). 
When all contraction ratios are equal, $r_j \equiv r$, we call the IFS \emph{homogeneous}. 
A homogeneous IFS of cardinality 2 with positive contraction ratio yields a measure $\mu$ called \textit{Bernoulli convolution} (without loss of generality, we may assume the translations are $b_j = 0$ or $1$) , which is one of the most studied and important examples in the literature, see for example \cite{varju2018recent} and references therein.

    Recall the following fundamental problem in fractal geometry : given a fixed $K \subset \mathbb{R}^n$, what can be said about the structure of its generating IFSs, with or without separation conditions? Feng and Wang \cite{FengWang2009} initiated this study for $K \subset \mathbb{R}$ under the separation condition OSC. They proved that all homogeneous generating IFSs with the OSC form a finitely-generated semigroup (if non-empty) when equipped with the composition, and also gave some sufficient conditions for these semigroups to have a minimal element (with or without homogeneity). They also provided some special examples where a minimal generating IFS does not exist. Later, Deng and Lau generalised the finitely-generated property for $K \subset \mathbb{R}^n$ under homogeneity and the separation condition SSC in \cite{DengLau2013}, and then relaxed the SSC to the OSC in \cite{DengLau2017}. Some further results on specific classes of self-similar sets were given, for example, on two connected fractals \cite{YaoLi2016} and on a construction with complete overlaps \cite{KongYao2021}.

    In this paper, we establish analogue results for self-similar measures on the real line and their generating weighted IFSs. A notable difference is that, in the homogeneous case, we do not require separation conditions in some results. We need some natural analogue definitions to \cite{FengWang2009}.

Let $(\{S_j\}_{j=1}^{N},(p_1,\dots,p_N))$  and $(\{T_i\}_{i=1}^{n},(q_1,\dots,q_n))$  be two weighted IFSs, their composition is defined as
$$(\{S_j\},(p_1,\dots,p_N))\circ (\{T_i\},(q_1,\dots,q_n))=(\{S_j\circ T_i\},(p_jq_i)_{i,j}).$$
There are infinitely many weighted IFSs that yields a same self-similar measure, for example, using this composition procedure. Denote by $\mathcal{M}(\mu)$ the set of all generating (standard) weighted IFSs of $\mu$, and denote by $\mathcal{M}_{\mathrm{P}%
}(\mu)\subset \mathcal{M}(\mu)$ the set of all generating weighted IFSs of $\mu$ satisfying the same property $%
\mathrm{P}$. In this paper, we concern two properties. One is homogeneity, denoted by P=HOM. Another is a separation condition P=SSC, which means that the union is disjoint in \eqref{at}.

    We present our main results for the P=HOM case.
\begin{theorem}\label{thm0}
Let $\Phi$ be a homogeneous IFSs in $\mathbb{R}$ with cardinality 2. Let $\mu$ be the self-similar measure generated by $(\Phi,\textbf{p})$ where $\textbf{p}$ is a probability vector that is not (0.5,0.5). Then $\mathcal{M}_{\mathrm{HOM}%
}(\mu)$ has a minimal element $(\Phi,\textbf{p})$, that is, all homogeneous generating weighted  IFSs of  $\mu$ is an iteration of  $(\Phi,\textbf{p})$.
\end{theorem}
\begin{theorem}\label{thm1}
Let $\Phi$ be a homogeneous IFSs in $\mathbb{R}$ with cardinality no less than 3, such that all its nonzero translations are linearly independent over $\mathbb{Q}$. Let $\mu$ be the self-similar measure generated by $(\Phi,\textbf{p})$ where $\textbf{p}$ is a probability vector. Then $\mathcal{M}_{\mathrm{HOM}%
}(\mu)$ has a minimal element $(\Phi,\textbf{p})$, that is, all  homogeneous generating weighted  IFSs of  $\mu$ is an iteration of  $(\Phi,\textbf{p})$.
\end{theorem}

    These two theorem show that, `most' homogeneous self-similar measures on the line have a minimal generating weighted IFS, since in the parameter space formed by translations (or probability vector), Lebesgue almost all choices satisfies the linearly independency (or not the Lebesgue measure). There are homogeneous self-similar measures on the line that do not have a minimal generating weighted IFS, for example, the Lebesgue measure on [0,1].

    The proof of these theorems make use of the following theorem. Given a weighted homogeneous IFS $(\{S_j\}_{j=1}^N,(p_1,\dots,p_N))$, we call \[
m(\xi) = \sum_{j=1}^N p_j \, e^{-2\pi i b_j \xi}
\] its \textit{corresponding exponential polynomial}.  An \emph{exponential polynomial} is a finite sum of the form
\[
f(z)=\sum_{j=1}^{n} a_j e^{\alpha_j z},
\]
where the coefficients $a_j$ and the frequencies $\alpha_j$ are complex numbers. Clearly, any exponential polynomial with real frequencies and normalized positive coefficients (that is, the sum of coefficients is 1) combined with a common contraction ratio $r$ uniquely corresponds to a weighted homogeneous IFS. We say that a homogeneous weighted IFS $(\Phi,\textbf{p})$ satisfies \textit{condition (Z)} if, either its attractor satisfies the `no-interval condition' in \cite[Before Lemma 5.1]{FengWang2009}, or its corresponding exponential polynomial has a complex zero that is not purely imaginary. We say that a measure $\mu$ satisfies \textit{condition  (HLC)}, short for `homogeneous logarithmic commensurability', if the absolute values of the common contraction ratios of IFSs in $\mathcal{M}_{\mathrm{HOM}}(\mu)$ are rational powers of each other.
\begin{theorem}\label{thm2}
If a homogeneous weighted IFS $(\Phi,\textbf{p})$ satisfies condition (Z), then its self-similar measure $\mu$ satisfies condition  (HLC).
\end{theorem}

    In condition (Z), the `no-interval' condition is automatically satisfied when the absolute value of the contraction ratio is small, while the `zero condition' does not rely on the contraction ratio at all. Certain conditions must be required, like condition (Z), to guarantee (HLC).
With more effort, one can show that, when the measure $\mu$ satisfies condition  (HLC), $\mathcal{M}_{\mathrm{HOM}%
}(\mu)$ is a finitely generated semigroup. Since Moran equation is not available, new ingredient is required compared with \cite{FengWang2009}. But this `finitely generated' property might fail without (HLC). To see this, just consider the Lebesgue measure on [0,1]. There are also further counterexamples based on convolutions of the Lebesgue measure on different intervals.

    Our result for the inhomogeneous P=SSC case heavily relies on that of self-similar sets. We say that a weighted IFS $(\Psi,\textbf{q})$ is \textit{derived from} $(\Phi,\textbf{p})$, if each contraction in $\Psi$ is in the form $\phi_{w}:=\phi_{w_1}\circ...\circ\phi_{w_m}$ with associated probability $p_{w}:=p_{w_1}...p_{w_m}$ for some word $w=w_1...w_m$ associated with $\Phi$, and $\Psi$ shares the same attractor with $\Phi$.
\begin{theorem}\label{thm3}
Let $\Phi$ be an IFSs in $\mathbb{R}$ satisfying the SSC with attractor $K$, such that each generating IFS of $K$ with the SSC is derived from $\Phi$. Let $\mu$ be the self-similar measure generated by $(\Phi,\textbf{p})$ where $\textbf{p}$ is a probability vector. Then $\mathcal{M}_{\mathrm{SSC}}(\mu)$ has a minimal element $(\Phi,\textbf{p})$, that is, all generating weighted IFSs of  $\mu$ satisfying the SSC is derived from $(\Phi,\textbf{p})$.
\end{theorem}     The condition `each generating IFS of $K$ with the SSC is derived from $\Phi$' is fulfilled under some easily checkable conditions, see for example, \cite[Theorem 4.1]{FengWang2009}. We remark that, without homogeneity or separation conditions, the structure of $\mathcal{M}(\mu)$ is always complicated, see an example in Section 4.

    The structure of this paper is as follows. We first prove Theorem \ref{thm2} in Section 2. Then, we prove Theorems \ref{thm1} and \ref{thm0} in Section 3. Finally, we prove Theorem \ref{thm3} and present an example in Section 4.

\section{Proof of Theorem \ref{thm2}}
Recall that the Fourier transform of a finite Borel measure $\mu$ on $\mathbb{R}$ is defined by
\begin{equation}\label{eq:Fourier}
    \widehat{\mu}(\xi) = \int_{\mathbb{R}} e^{-2\pi i x \xi} \, d\mu(x), \qquad \xi \in \mathbb{R}.
\end{equation}
For a self-similar measure satisfying \eqref{eq:inv}, one readily obtains the functional equation
\begin{equation}\label{eq:FTeq}
    \widehat{\mu}(\xi) = \sum_{j=1}^{N} p_j \, e^{-2\pi i b_j \xi} \; \widehat{\mu}(r_j \xi).
\end{equation}
Iterating \eqref{eq:FTeq} leads to explicit representations of $\widehat{\mu}(\xi)$. In the homogeneous case where $r_j = r$ for all $j$, one obtains the classical infinite product expansion
\begin{equation}\label{eq:prod}
    \widehat{\mu}(\xi) = \prod_{k=0}^{\infty} \Bigg( \sum_{j=1}^{N} p_j \, e^{-2\pi i b_j r^{k} \xi} \Bigg),
\end{equation}
which is a uniformly convergent infinite product on compact subsets of $\mathbb{R}$, making it a Riesz-type product.

We will analyse the zero sets of the Fourier transform on the complex plane. Given a complex function $f$ defined on the complex plane, denote by $Z(f)$ the set of its zeros (on the complex plane).
\begin{lemma}\label{2.0}Let $\mu$ be the self-similar measure generated by a homogeneous weighted IFSs $(\Phi,\textbf{p})$ in $\mathbb{R}$ with contraction ratio $r$, and $m(\xi)$ being the corresponding exponential polynomial. Then the set of the zeros of $\widehat{\mu}$ on the complex plane is $$Z(\widehat{\mu}):=\bigcup_{k=0}^{\infty} r^{-k} Z(m).$$
\end{lemma}
\begin{proof}This directly follows from \eqref{eq:prod}.\end{proof}

The following important logarithmic commensurability lemma is required.
\begin{lemma}\label{2.1}
Let \(A\) be a finite set of real numbers, and let \(B\) be an arbitrary set of real numbers. Let \(a,b > 1\) and assume
\[
S = \bigcup_{k=0}^{\infty} a^k A = \bigcup_{k=0}^{\infty} b^k B .
\]
If \(S\) contains a non-zero element, then \(b\) is a rational power of \(a\).
\end{lemma}

\begin{proof}
Define \(S^+ = S \cap (0,\infty)\) and \(S^- = S \cap (-\infty,0)\); at least one of them is non-empty.  If \(S^+ \neq \emptyset\) we apply the argument below to \(S^+\). If \(S^+ = \emptyset\) then \(S^- \neq \emptyset\); replacing \(A,B\) by \(-A,-B\) leaves \(a,b\) and \(|A|\) unchanged and turns \(S^-\) into a set of positive numbers. Thus we may assume without loss of generality that \(S^+ \neq \emptyset\).

By multiplying all elements of \(A\) and \(B\) by a suitable positive constant (which does not affect \(a,b\) nor the union equality) we may also assume \(\min S^+ = 1\).

Set \(\alpha = \log a > 0\), \(\beta = \log b > 0\). Define \(X = \log S^+ = \{\log s : s \in S^+\}\).
Then \(X \subseteq [0,\infty)\), \(0 \in X\), and
\[
X = \log(A^+) + \mathbb{Z}_{\ge 0}\alpha = \log(B^+) + \mathbb{Z}_{\ge 0}\beta,
\]
where \(A^+ = A \cap (0,\infty)\) and \(B^+ = B \cap (0,\infty)\).

Consider the quotient map \(\pi \colon \mathbb{R} \to \mathbb{R}/\alpha\mathbb{Z}\).
Since \(X = \log(A^+) + \mathbb{Z}_{\ge 0}\alpha\), every element of \(X\) is congruent modulo \(\alpha\) to some element of \(\log(A^+)\).
Hence \(\pi(X) = \pi(\log(A^+))\) is finite.

On the other hand, from \(X = \log(B^+) + \mathbb{Z}_{\ge 0}\beta\) we see that for any \(x \in X\) the whole semi-orbit
\(x + \mathbb{Z}_{\ge 0}\beta\) is contained in \(X\). Therefore
\[
\pi(x + \mathbb{Z}_{\ge 0}\beta) \subseteq \pi(X).
\]
But \(\pi(x + \mathbb{Z}_{\ge 0}\beta) = \{ \pi(x) + k\beta \bmod \alpha : k \in \mathbb{Z}_{\ge 0} \}\).

If \(\beta/\alpha \notin \mathbb{Q}\), then the set \(\{k\beta \bmod \alpha : k \ge 0\}\) is infinite (it is dense in the closed subgroup it generates, which must be infinite for an irrational rotation).  Consequently \(\pi(x + \mathbb{Z}_{\ge 0}\beta)\) would be infinite, contradicting \(|\pi(X)| < \infty\).  Thus \(\beta/\alpha \in \mathbb{Q}\), which completes the proof.
\end{proof}

We are now in a position to prove \ref{thm2}. In the proof we will use the distribution of complex zeros of exponential polynomials, and consider the intersection with a certain line.
\begin{proof}
The proof of (HLC) under the no-interval condition is given in \cite[Lemma 5.1]{FengWang2009}. It remains to prove that, if the corresponding exponential polynomial $m(\xi)$ has a complex zero $z$ that is not purely imaginary, then (HLC) holds.

Indeed, consider the line $L$ connecting 0 and $z$ on the complex plane. By \cite[Theorem 3.6]{lapidus2013}, the zeros of $m(\xi)$ have bounded real parts, thus the set of zeros of $m(\xi)$ located on $L$ is finite (the zeros of analytical functions are discrete), denoted by $A$. Assume that the common ratio of $(\Phi,\textbf{p})$ is $a^{-1}$. Let $f(\xi)$ be the corresponding exponential polynomial of any weighted IFS $(\Psi,\textbf{q})$ with common ratio $b^{-1}$  in $\mathcal{M}_{\mathrm{HOM}%
}(\mu)$. By Lemma \ref{2.0}, $$Z(\widehat{\mu})=\bigcup_{k=0}^{\infty} a^{k} Z(m)=\bigcup_{k=0}^{\infty} b^{k} Z(f),$$ thus denote  $Z_L=Z(\widehat{\mu})\bigcap L$ and $B=Z(f)\bigcap L$, we have
$$Z_L=\bigcup_{k=0}^{\infty} a^k A = \bigcup_{k=0}^{\infty} b^k B .$$
If $a,b$ contains negative numbers, just consider $(\Phi,\textbf{p})\circ (\Phi,\textbf{p})$ and $(\Psi,\textbf{q})\circ(\Psi,\textbf{q})$, and this does not change the (HLC) property. Then we may assume  \(a,b > 1\). By Lemma \ref{2.1}, (HLC) holds true. The proof is complete.

\end{proof}

\section{Proof of Theorems \ref{thm1} and \ref{thm0}}
We first prove Theorem \ref{thm1} , since Theorem \ref{thm0} follows from an easier similar routine.
We need to verify condition (Z). We may always reduce an exponential polynomial to the following form by multiplying some suitable $ce^{\alpha z}$, which has no complex zero (and nothing essentially changes).
\begin{lemma}\label{3.0}
Let $f(z)=1+\sum_{j=1}^{n} a_j e^{\alpha_j z}$ be an exponential polynomial with $a_j\in R\setminus\{0\}$ and $\alpha_j$ be positive numbers, $n\geq 2$. If $\{\alpha_1,\dots,\alpha_n\}$ are linearly independent over $ Q$, then not all complex zeros of $f$ share a same real part.
\end{lemma}
\begin{proof}This is the `generic nonlattice case' in \cite[Theorem 3.6]{lapidus2013}. A direct computation shows that $$D_l\neq D_r$$ in \cite[Theorem 3.6 (3.14), (3.15)]{lapidus2013}, which gives the desired.\end{proof}

Next, we prove that the corresponding exponential polynomial is irreducible. Before this, we need a basic proposition.
\begin{proposition}
Let $N\ge 2$, let $k_1,k_2,\dots,k_N$ be non‑negative integers with at least two of them positive,
and let $a_1,a_2,\dots,a_N$ be non‑zero real numbers.
Then the polynomial
\[
f(x_1,x_2,\dots,x_N)=\sum_{i=1}^{N} a_i\, x_i^{k_i}
\]
cannot be written as a product of two non‑unit elements in the ring $R[x_1^{\pm1},\dots,x_N^{\pm1}]$ of real Laurent polynomials (allowing negative integer powers).
\end{proposition}

\begin{proof}
Assume, for a contradiction, that $f=GH$ with $G,H\in R[x_1^{\pm1},\dots,x_N^{\pm1}]$ both non‑units (i.e.\ neither $G$ nor $H$ is a monomial). Clearly, $G$ and $H$ can be taken to be ordinary polynomials as well, since otherwise $GH$ would tend to infinity if we let one variable (with negative powers) tend to 0 and fix the others.
It follows that there would be terms in $f$ containing at least two different variables, a contradiction!
\end{proof}

The set of all exponential polynomials, equipped with pointwise addition and multiplication, will be denote by $EP$. Factorization theory in $EP$ was initiated by J.\,F.\ Ritt~\cite{ritt1927} and has connections with difference
algebra and transcendental number theory. A non-zero element $g\in EP$ is \emph{irreducible} if it is not a unit (the units are precisely the nowhere-zero functions $c e^{\beta z}$ with $c\neq0$) and cannot be expressed as a product of
two non-units in $EP$.

\begin{lemma}\label{3.2}
Let $f(z)=1+\sum_{j=1}^{n} a_j e^{\alpha_j z}$ be an exponential polynomial with $a_j\in R\setminus\{0\}$ and $\alpha_j$ be positive numbers, $n\geq 2$. If $\{\alpha_1,\dots,\alpha_n\}$ are linearly independent over $ Q$, then $f$ is irreducible in the ring $EP$.
\end{lemma}

\begin{proof}
 Let $y_j= e^{\alpha_j z}$, and define \[
Q(y_1,\dots,y_m) = 1 + \sum_{j=1}^{m} a_j y_j.
\]According to Ritt’s first theorem in \cite{ritt1927}, since the frequencies are linearly independent over \(\mathbb{Q}\), every non‑trivial factor of $f$ arises from some positive integers \(t_1,\dots,t_m\) such that the polynomial
\[
R(y_1,\dots,y_m) = Q(y_1^{t_1},\dots,y_m^{t_m}) = 1 + \sum_{j=1}^{m} a_jy_j^{t_j}
\]
admits a decomposition into non‑constant Laurent polynomials with constant term 1. By the above proposition, $Q$ is irreducible, showing the desired.
\end{proof}

We are now in a position to prove Theorem \ref{thm1}.
\begin{proof}
Let  $(\Psi,\textbf{q})\in \mathcal{M}_{\mathrm{HOM}%
}(\mu)$ .  Denote by $a,b$ the common ratio of  $\Phi,\Psi$ respectively.  Denote by $f,g$ the corresponding exponential polynomial of $(\Phi,\textbf{p})$ and $(\Psi,\textbf{q})$. Theorem \ref{thm2} and Lemma \ref{3.0} guarantee the property (HLC). We may find positive integers $p,q$ satisfying $(p,q)$=1 or 2 such that $a^p=b^q$. Then by using \eqref{eq:prod}, we obtain
$$f(x)f(ax)...f(a^{p-1}x)=g(x)g(bx)...g(b^{q-1}x).$$
By Ritt's unique factorization theorem in \cite{ritt1927}, $g$ has no simple factor, and denote by $k$ the cardinality of its irreducible factors (multiplicity taken into account). Since $f$ is irreducible by assumption and Lemma \ref{3.2}, we have $$p=kq.$$
When $(p,q)$=1, we know that $q=1$ and so $p=k$, thus $$g(x)=f(x)f(ax)...f(a^{p-1}x)$$ and so $(\Psi,\textbf{q})$ is a $p$th iteration of $(\Phi,\textbf{p})$.
When $(p,q)$=2, we know that $q=2$ and so $p=k$. The result is the same when $b=a^k$, and it is impossible that $b=-a^k$. To see this, otherwise, we would have $$g(x)g(-a^kx)=f(x)f(ax)...f(a^{2k-1}x).$$ It follows that $$g(x)=e^{Cx}f(x)...f(a^{k-1}x), g(-a^kx)=e^{-Cx}f(a^{k}x)...f(a^{2k-1}x)$$ for some constant $C$, which further implies that
$$f(x)=e^{cx}f(-x)$$ for some constant $c$. This would imply that the frequencies of $f$ form a finite symmetric set, a contradiction to the linear independency!
The proof is complete.
\end{proof}

We prove Theorem  \ref{thm0} in a similar way.
\begin{proof}
Let  $(\Psi,\textbf{q})\in \mathcal{M}_{\mathrm{HOM}%
}(\mu)$ .  Denote by $a,b$ the common ratio of  $\Phi,\Psi$ respectively.  Denote by $f,g$ the corresponding exponential polynomial of $(\Phi,\textbf{p})$ and $(\Psi,\textbf{q})$. Theorem \ref{thm2} and a direct verification that $$f=1-p+pe^z$$
has a complex zero that is not purely imaginary guarantee the property (HLC). We may find positive integers $p,q$ satisfying $(p,q)$=1 or 2 such that $a^p=b^q$. Then by using \eqref{eq:prod}, we obtain
$$f(x)f(ax)...f(a^{p-1}x)=g(x)g(bx)...g(b^{q-1}x).$$
By Ritt's unique factorization theorem in \cite{ritt1927}, $g$ has only simple factors, and denote by $k$ the cardinality of its irreducible factors (multiplicity taken into account). Since $f$ itself is a simple factor, by assumption and Lemma \ref{3.2}, we have $$p=kq.$$
The rest of the proof is virtually identical. To see that $b=-a^k$ is impossible, just note that $$f(x)=e^{cx}f(-x)$$ for some constant $c$ would imply that the probability vector is (0.5,0.5). The proof is complete.
\end{proof}

\section{The inhomogeneous case and some examples}
The proof of Theorem \ref{thm3} is very short.
\begin{proof}Let $(\Psi,\textbf{q})$ be a weighted IFS in $\mathcal{M}_{\mathrm{SSC}}(\mu)$. Then the support of $\mu$ is $K$, thus $\Psi$ is a generating IFS of $K$ with the SSC. By assumption, $\Psi$ is derived from $\Phi$, and they share the same attractor $K$. For each map $\phi_{w}$ in $\Psi$, where $w$ is a word associated with $\Phi$, we consider its associated probability $q_{w}$. Since $\Psi$ satisfies the SSC, $$q_{w}=\mu(\phi_{w}(K))=p_{w},$$ showing the desired. \end{proof}

The following example on middle third Cantor set shows the complexity of the structrue of $\mathcal{M}(\mu)$. Thus separation conditions must be required to guarantee the existence of a minimal element. With more effort, the following example can be generalised to characterize the elements in $\mathcal{M}(\mu)$. In particular, under the same condition of Theorem \ref{thm3}, the structure of $\mathcal{M}(\mu)$ could be fully characterized, though the statement is inevitably complicated.
\begin{example}
Let $\mu$ be the self-similar measure generated by the IFS
\[
f(x)=\frac{1}{3}x,\qquad g(x)=\frac{1}{3}x+\frac{2}{3}
\]
with probabilities $p$ and $1-p$ ($0<p<1$).
Consider the IFS consisting of the four maps
\[
F_1=f,\quad F_2=f\circ f,\quad F_3=f\circ g,\quad F_4=g,
\]
together with the corresponding probabilities
\[
q,\quad p(p-q),\quad (1-p)(p-q),\quad 1-p,
\]
where $0\le q\le p\le 1$. Let $\nu$ be its self-similar measure. We prove that $\nu=\mu$.

The Fourier transform $\hat\mu(\xi)=\int e^{-2\pi i\xi x}\,d\mu(x)$ satisfies
\begin{equation}\label{eq:mu}
\hat\mu(\xi)=\hat\mu\!\left(\frac{\xi}{3}\right)\Bigl[p+(1-p)e^{-4\pi i\xi/3}\Bigr].
\end{equation}
The contraction ratios and translations of the new maps are
\[
F_1(x)=\frac{x}{3},\;
F_2(x)=\frac{x}{9},\;
F_3(x)=\frac{x}{9}+\frac{2}{9},\;
F_4(x)=\frac{x}{3}+\frac{2}{3}.
\]
Hence any self-similar measure $\nu$ for the new system must obey
\begin{align}\label{eq:nu}
\hat\nu(\xi)&=q\,\hat\nu\!\left(\frac{\xi}{3}\right)
            +p(p-q)\,\hat\nu\!\left(\frac{\xi}{9}\right)
            +(1-p)(p-q)\,\hat\nu\!\left(\frac{\xi}{9}\right)e^{-4\pi i\xi/9}
            +(1-p)\,\hat\nu\!\left(\frac{\xi}{3}\right)e^{-4\pi i\xi/3}\notag\\
&=\hat\nu\!\left(\frac{\xi}{3}\right)\bigl[q+(1-p)e^{-4\pi i\xi/3}\bigr]
  +\hat\nu\!\left(\frac{\xi}{9}\right)(p-q)\bigl[p+(1-p)e^{-4\pi i\xi/9}\bigr].\tag{2}
\end{align}
Replace $\xi$ by $\xi/3$ in \eqref{eq:mu} gives
\begin{equation}\label{eq:shift}
\hat\mu\!\left(\frac{\xi}{3}\right)=\hat\mu\!\left(\frac{\xi}{9}\right)\bigl[p+(1-p)e^{-4\pi i\xi/9}\bigr].
\end{equation}
Now substitute $\hat\mu$ for $\hat\nu$ in the right‑hand side of \eqref{eq:nu} gives
\begin{align*}
\text{RHS}
&=\hat\mu\!\left(\frac{\xi}{3}\right)\bigl[q+(1-p)e^{-4\pi i\xi/3}\bigr]
  +\hat\mu\!\left(\frac{\xi}{9}\right)\bigl[p+(1-p)e^{-4\pi i\xi/9}\bigr]\\
&=\hat\mu\!\left(\frac{\xi}{3}\right)\bigl[q+(1-p)e^{-4\pi i\xi/3}+(p-q)\bigr]\\
&=\hat\mu\!\left(\frac{\xi}{3}\right)\bigl[p+(1-p)e^{-4\pi i\xi/3}\bigr]=\hat\mu(\xi).
\end{align*}
Thus $\hat\mu$ exactly satisfies the functional equation \eqref{eq:nu}, showing the desired.
\end{example}

\end{document}